\documentclass{article}
\usepackage[utf8]{inputenc}
\usepackage[english]{babel}
 
\usepackage[square,sort,comma,numbers]{natbib}
\usepackage{pigpen}
\newcommand{\po}{\ar@{}[dr]|{\text{\pigpenfont R}}}
\newcommand{\pb}{\ar@{}[dr]|{\text{\pigpenfont J}}}
\usepackage{graphicx}
\usepackage{tikz}
\usetikzlibrary{arrows}
\usetikzlibrary{matrix,arrows,decorations.pathmorphing}

\usepackage{graphicx}
\usepackage{tikz}
\usetikzlibrary{arrows}
\usetikzlibrary{matrix,arrows,decorations.pathmorphing}

\usepackage{amsmath}
\usepackage{amssymb}
\usepackage{amsthm}
\usepackage{float}

\numberwithin{equation}{section}
\numberwithin{figure}{section}

\usepackage{stmaryrd}
\usepackage[all]{xy}
\usepackage[mathcal]{eucal}
\usepackage{verbatim} 
\usepackage{fullpage} 
\usepackage{wrapfig}
\usepackage{lipsum}
\usepackage[all]{xy}
\theoremstyle{plain}
\newtheorem{thm}{Theorem}[section]

\newtheorem{lem}[thm]{Lemma}
\newtheorem{prop}[thm]{Proposition}
\newtheorem{cor}[thm]{Corollary}

\newtheorem{question}[thm]{Question}

\newtheorem{remark}[thm]{Remark}

\theoremstyle{definition}
\newtheorem{defn}[thm]{Definition}

\theoremstyle{remark}
\newtheorem*{rem}{Remark}

\title{The universal $n$-pointed surface bundle only has $n$ sections}
\author{Lei Chen}

\begin{document}
 \bibliographystyle{alpha}
\maketitle

\begin{abstract}
The classifying space BDiff$(S_{g,n})$ of the orientation-preserving diffeomorphism group of a surface $S_g$ of genus $g>1$ fixing $n$ points pointwise has a universal bundle \[
S_g \to \text{UDiff}(S_{g,n})\xrightarrow{u_{g,n}}\text{BDiff}(S_{g,n}).
\] 
The $n$ fixed points provide $n$ sections $\{s_i |1\le i\le n\}$ of $u_{g,n}$. In this paper we prove a conjecture of R. Hain that any section of $\pi$ is homotopic to some $s_i$. Let $\text{PConf}_n(S_g)$ be the space of ordered $n$-tuple of distinct points on $S_g$. As part of the proof of Hain's conjecture, we prove a result of independent interest: any surjective homomorphism $\pi_1(\text{PConf}_n(S_g))\to \pi_1(S_g)$ is equal to one of the forgetful homomorphisms $\{p_i:\pi_1(\text{PConf}_n(S_g))\to \pi_1(S_g)| 1\le i\le n\}$,  possibly post-composed with an automorphism of $\pi_1(S_g)$. We also classify sections of the universal hyperelliptic surface bundle.
	\end{abstract}
	
	\tableofcontents
\section{Introduction}

Let Diff$(S_{g,n})$ be the orientation-preserving diffeomorphism group of a surface $S_g$ of genus $g>1$ fixing $n$ distinct points $\{x_1,x_2,...,x_n\}\subset S_g$ pointwise. There is a fiber bundle
\begin{equation}
S_g \to \text{UDiff}(S_{g,n})\xrightarrow{u_{g,n}}\text{BDiff}(S_{g,n}),
\label{uni1}
\end{equation}
which is universal in the sense that any $S_g$-bundle endowed with $n$ disjoint sections is a pullback of this bundle. Since Diff$(S_{g,n})$ fixes the $n$ points $x_1,x_2,...,x_n$, we associate $n$ points on each fiber, i.e. n disjoint sections of (\ref{uni1}) which are denoted by $s_1,s_2,...,s_n$. A natural question is: are there more sections?

R. Hain conjectured that every section of (\ref{uni1}) is homotopic to one of these $n$ sections. This is the main theorem of this paper.

\begin{thm}[{\bf The classification of sections for ordered case}]
For $n\ge 0$ and $g>2$, every section of the universal bundle (\ref{uni1}) is homotopic to $s_i$ for some $i\in \{1,2,...,n\}$.  For $g=2$, there are precisely $2n$ homotopy classes of sections of the universal bundle (\ref{uni1}).
\label{main1}
\end{thm} 
Since each section $s_i$ has nontrivial self-intersection, we have the following corollary.
\begin{cor}
The universal bundle (\ref{uni1}) does not admit $n+1$ disjoint sections.
\label{nolift}
\end{cor}
What if we only fix the $n$ points as a set? More precisely, let Diff$(S_{g,\overline{n}})$ denote the orientation-preserving diffeomorphism group of a surface $S_g$ of genus $g>1$ fixing $n$ points $\{x_1,x_2,...,x_n\}\subset S_g$ as a set. There is a fiber bundle
\begin{equation}
S_g \to \text{UDiff}(S_{g,\overline{n}})\xrightarrow{u_{g,n}'} \text{BDiff}(S_{g,\overline{n}}).
\label{uni2}
\end{equation}
We also have the following result.
\begin{thm}[\bf No sections for unordered case]
For $n>1$ and $g>1$, surface bundle (\ref{uni2}) has no sections.
\label{main2}
\end{thm}
We see below that Hain's conjecture can be interpreted both in terms of mapping class groups and also in terms of moduli spaces. Let $\CMcal{M}_{g,m,n}$ be the moduli space of smooth Riemann surfaces of genus $g$ with $m+n$ distinct points, $m$ labelled and $n$ unlabelled. Earle-Kra \cite[Theorem 2.2]{EK} proved that the only holomorphic section of the forgetful map $f : \CMcal{M}_{g,m,n} \to \CMcal{M}_{g,m,0}$ occurs when $g=2$ and $n=6$. This section is constructed by marking all six Weierstrass points.

Corollary \ref{nolift} and Theorem \ref{main2} give a topological proof of the fact that there is no continuous section of $\CMcal{M}_{g,m+1,0}\to \CMcal{M}_{g,m,0}$ for $m\ge 0$ and there is no continuous section of $\CMcal{M}_{g,1,n}\to \CMcal{M}_{g,0,n}$ for $n>1$. Recently, we found out that Theorem \ref{main1} can be deduced from \cite[Theorem 1.1]{suoto}. Their proof substantially uses the tool of canonical reduction system. We provide a more elementary proof of this result.

When we talk about fundamental group in this paper, we omit the base point and that brings no ambiguity.

\subsection{The strategy of proof}
Let $\text{PConf}_n(S_g)$ be the space of ordered $n$-tuple of distinct points on $S_g$ and let $PB_n(S_g)=\pi_1(\text{PConf}_n(S_g))$. Let Mod$_{g,n}$ (resp. PMod$_{g,n}$) be the \emph{mapping class group} of $S_{g,n}$, i.e. the group of isotopy classes of orientation-preserving diffeomorphisms of $S_{g,n}$ fixing n punctures as a set (resp. pointwise). We omit $n$ when $n=0$.

We first translate the problem into a group-theoretical problem of determining a homomorphism $p$ satisfying the following diagram, where the horizontal exact sequences are the \emph{Birman exact sequences}.

  \begin{equation}
\xymatrix{
1 \to PB_n(S_g)\ar[r]\ar@{~)}[d]^R & \text{PMod}_{g,n} \ar[r]^{\pi_{g,n}}\ar@{~)}[d]^p& \text{Mod}_g\ar[r]\ar[d]^=&  1 \\
1 \to\pi_1(S_g)\ar[r] & \text{Mod}_{g,1}\ar[r]^{\pi_1}& \text{Mod}_g \ar[r]& 1.}              
\label{dia}
\end{equation}

The analysis of $p$ is decomposed into two parts: first classifying $R$ and then trying to extend $R$ to $p$. In the second part, we use the commutativity of diagram (\ref{dia}) and the action of Mod$_g$ on $\pi_1(S_g)$. In classifying $R$, we have the following key ingredient.

\subsection{The key ingredient}
The key ingredient is the following question. 
\begin{question}
How many homotopy classes of maps are there from $\text{PConf}_n(S_g)$ to $S_g$?\label{question}
\end{question}

Let $p_i:\text{PConf}_n(S_g)\to S_g$ be the projection onto the $i$th component. Let $p_{i*}:PB_n(S_g)\to \pi_1(S_g)$ be the map on the fundamental groups of $p_i$. Since $p_i$ does not fix a basepoint, the map $p_{i*}$ is only defined up to conjugacy. Do we have more maps?

\begin{rem}
The following figure is a cartoon version of what the following theorem talks about.
\end{rem}

\begin{figure}[!htb]
\minipage{0.45\textwidth}
  \includegraphics[width=\linewidth]{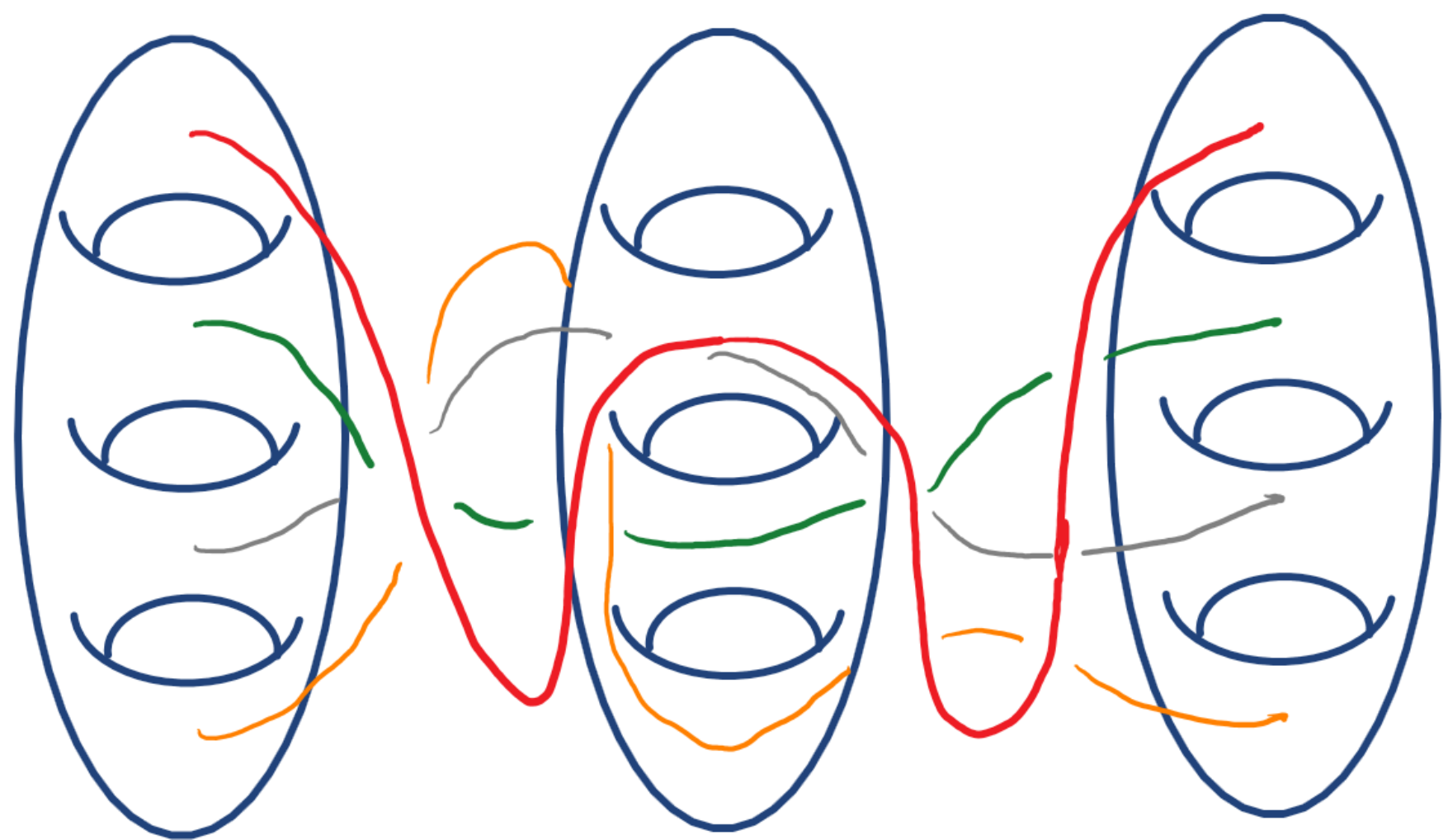}
\endminipage\hfill
  $\xrightarrow{\text{Forget}}$
\minipage{0.45\textwidth}
  \includegraphics[width=\linewidth]{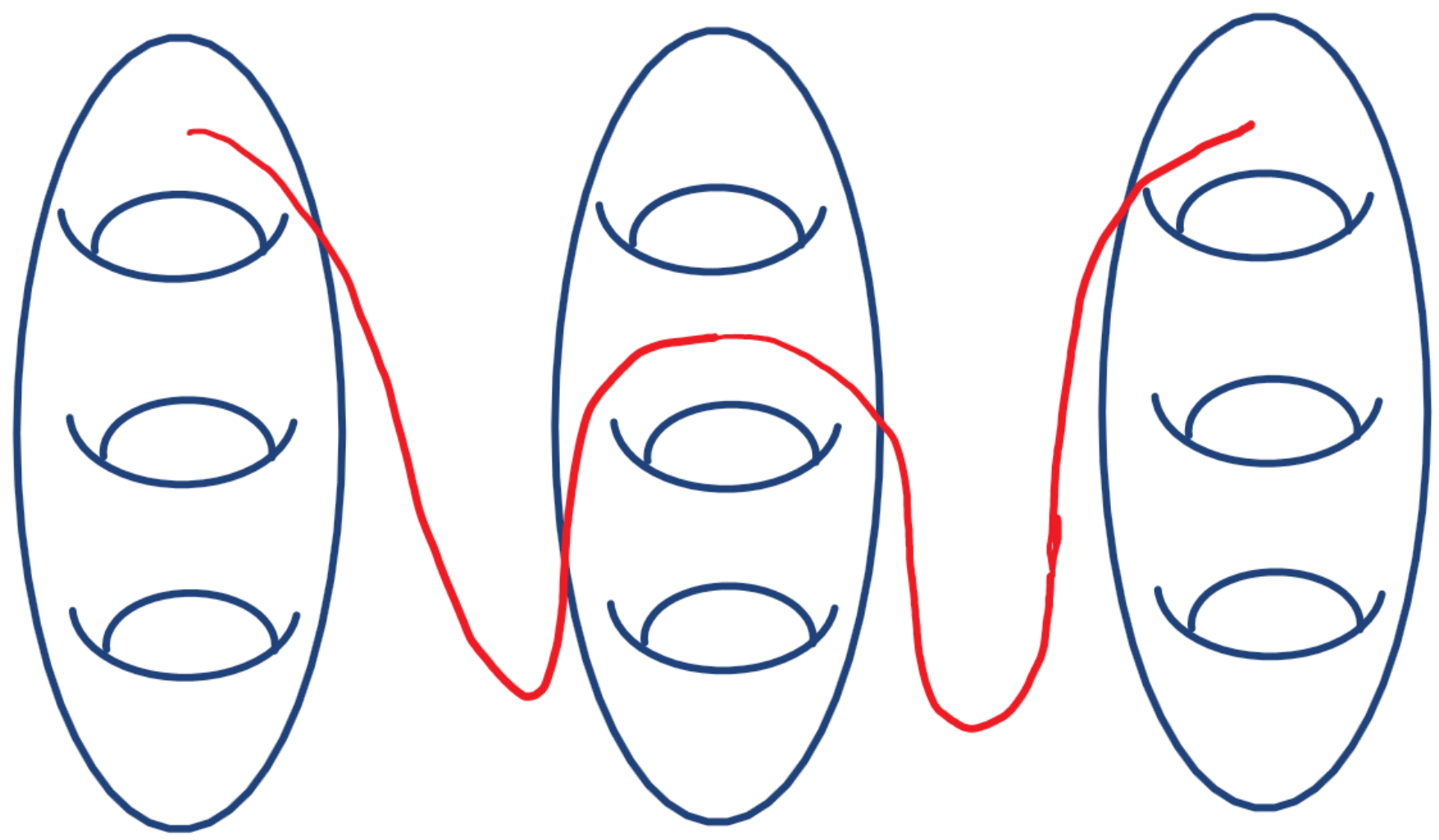}
\endminipage\hfill
\caption{A braid group homomorphism}
\end{figure}

We answer Question \ref{question} by the following classification theorem.

\begin{thm}[{\bf The classification of homomorphisms $PB_n(S_g)\to \pi_1(S_g)$}]
Let $g>1$ and $n>0$. Let $R: PB_n(S_g)\to \pi_1(S_g)$ be a homomorphism. The followings hold:\\
\\
1)If $R$ is surjective, then $R=A\circ p_{i*}$ for some $i$ and $A$ an automorphism of $\pi_1(S_g)$.\\ 
\\
2)If Image$(R)$ is not a cyclic group, the homomorphism $PB_n(S_g)\to \pi_1(S_g)$ factors through $p_{i*}$ for some $i$.
\label{sur}

\end{thm}
In our next paper \cite{lei3}, we classify the surjective homomorphisms between $PB_n(S_g)$ and $PB_m(S_g)$ for any $n$ and $m$. We also give a new proof of the result in \cite[Theorem 1]{AutPB} about the automorphism group of $PB_n(S_g)$.

\subsection{Other geometric applications}
It is a basic question to understand the classification of sections of a surface bundle. Theorem \ref{sur} has many geometric applications regarding the section problems. In this paper, we also deal with the case of the universal hyperelliptic surface bundle. This result is recently proved in \cite[Theorem 1]{Watanabe} as well. The genus 2 case in Theorem \ref{main1} is also part of the hyperelliptic case.\\
\\
\noindent
{\large\bf Acknowledgements}

The author would like to thank R. Hain for asking the question and providing the background for this problem. She would also like to thank Nick Salter for correcting the first draft of the paper and discussing the content. She thanks Paul Apisa, Jarred Sanders for comments on the paper and discussions on related problems. Lastly, she would like to extend her warmest thanks to Benson Farb for his extensive comments as well as his invaluable support from start to finish.

\section{The translation of the problem into a group-theoretical problem}
In this section, we translate the problem of finding a section of the universal surface bundle into a purely group-theoretical problem of finding homomophisms of mapping class groups. 
\subsection{The translation of the the section problem}
In this subsection, we will translate the problem of finding a section of a surface bundle into a purely group-theoretical problem of finding homomophisms of groups. 

Let Diff$(S_g)$ denote the orientation-preserving diffeomorphism group of a surface $S_g$ of genus $g>1$. We have the universal Diff$(S_g)$ principal bundle
\[
\text{Diff}(S_g)\to \text{EDiff}(S_g)\to \text{BDiff}(S_g).
\]
Here EDiff$(S_g)$ is the total space of the universal Diff$(S_g)$ bundle, i.e. a contractible principal Diff$(S_g)$ bundle. Let UDiff$(S_g)=\text{EDiff}(S_g)\times_{\text{Diff}(S_g)}S_g$ be the universal surface bundle
\[
S_g\to \text{UDiff}(S_g) \xrightarrow{u_g} \text{Diff}(S_g).
\]
$\text{BDiff}(S_g)$ classifies surface bundles, which means that any surface bundle $S_g\to E\to B$ is the pullback of $u_g$ via a continuous map $f_C: B\to \text{BDiff}(S_g)$.  Let Mod$_{g,n}$ (resp. PMod$_{g,n}$) be the \emph{mapping class group} (resp. \emph{pure mapping class group}) of $S_{g,n}$, i.e. the group of isotopy classes of orientation-preserving diffeomorphisms of $S_{g}$ fixing n punctures as a set (resp. pointwise). We omit $n$ when $n=0$.

Earle and Eells \cite[Theorem 1]{EE} says that Diff$_0(S_g)$, i.e. the identity component of Diff$(S_g)$, is contractible for $g>1$. Therefore we have $\text{BDiff}(S_g)=K(\text{Mod}_g,1)$. By the property of $K(\pi,1)$ space, $f: B\to \text{BDiff}(S_g)$ is determined by the monodromy representation
\[ f_*:\pi_1(B)\to \text{Mod}_g.\]

We have the following correspondence:
\newcommand{\pctext}[2]{\text{\parbox{#1}{\centering #2}}}
\begin{alignat}{2}
      \label{correspondence}
 \Biggl\{ \pctext{1.5in}{Conjugacy classes of representations $f: \pi_1(B)\to \text{Mod}_g$ }\Biggr\}                                                      
      \Longleftrightarrow \Biggl\{  \pctext{1.5in}{Isomorphism classes of oriented $S_g$-bundles over $B$} \Biggr\}.    
      \end{alignat}

Let $\text{Diff}(S_{g,1})$ be the orientation-preserving diffeomorphism group of a surface $S_g$ of genus $g>1$ fixing one point. There is a natural inclusion $\text{Diff}(S_{g,1})\hookrightarrow \text{Diff}(S_g)$.

\begin{prop}For $g>1$, 
\[
\text{UDiff}(S_g)=\text{BDiff}(S_{g,1}).
\]
\label{uni}
\end{prop}
\begin{proof}
\begin{align*}
\text{UDiff}(S_g)&= \text{EDiff}(S_g)\times_{\text{Diff}(S_g)}S_g && \text{By definition} \\
&= \text{EDiff}(S_g)\times_{\text{Diff}(S_g)} \text{Diff}(S_g)/\text{Diff}(S_{g,1}) && \text{Because } S_g=\text{Diff}(S_g)/\text{Diff}(S_{g,1}) \\
&= \text{EDiff}(S_g)/\text{Diff}(S_{g,1}) &&  \text{$\text{Diff}(S_{g,1})$ is a subgroup of Diff$(S_g)$}\\
&=\text{BDiff}(S_{g,1}). && \text{$\text{EDiff}(S_g)$ is contractible}
\end{align*}
\end{proof}
Proposition \ref{uni} implies that the universal surface bundle is
\begin{equation}
S_g\to K(\text{Mod}_{g,1},1)\to K(\text{Mod}_{g},1).
\label{UB}
\end{equation}
The fundamental groups of surface bundle (\ref{UB}) gives the following short exact sequence.
\begin{equation}
1\to \pi_1(S_g)\to \text{Mod}_{g,1}\to \text{Mod}_g\to 1.
\label{universal bundle}
\end{equation}

\begin{question}[\bf Section problems]
For a surface bundle $S_g\to E\xrightarrow{f} B$, denote by $\rho: \pi_1(B)\to \text{\normalfont Mod}_g$ the monodromy representation of $f$. The fundamental groups of surface bundle $f$ gives the following short exact sequence.
\begin{equation}
1\to \pi_1(S_g)\to \pi_1(E) \xrightarrow{f_*} \pi_1(B)\to 1.
\label{SP}
\end{equation}
How many splittings are there of exact sequence (\ref{SP})?
\end{question}

It is well-known that exact sequence (\ref{universal bundle}) has no splittings. This is $n=0$ case of Theorem \ref{main1}. The answer is no because of torsion, e.g. \cite[Corollary 5.11]{BensonMargalit}. The key fact is that there are noncyclic finite subgroups in Mod$_{g}$ but there does not exist noncyclic finite subgroups in Mod$_{g,1}$.

By the property of the pullback diagram, finding a splitting of $f_*$ is the same as finding a homomorphism $p$ making the following diagram commute, i.e. $\pi_1\circ p=\rho$.
\begin{equation}
\xymatrix{
\pi_1(E)\ar[r]\ar[d]^{f_*}& \text{Mod}_{g,1}\ar[d]^{\pi_{g,1}}\\
\pi_1(B)\ar[r]^{\rho}\ar[ru]^{p} & \text{Mod}_g.}
\label{CD}
\end{equation}

For a surface bundle $S_g\to E\xrightarrow{f} B$, we have the following correspondence:
\begin{alignat}{2}
      \label{correspondence2}
 \Biggl\{ \pctext{2in}{Homotopy classes of continuous sections of $S_g\to E\xrightarrow{f} B$}\Biggr\}                                                      
      \Longleftrightarrow \Biggl\{  \pctext{3in}{Homomorphisms $p$ satisfying diagram (\ref{CD}) up to an conjugacy by an element in Ker$(\pi_{g,1})\cong \pi_1(S_g)$} \Biggr\}.    
      \end{alignat}

\begin{remark}
The conjugation is needed here for the lack of base points on the spaces. The classification of homomorphisms of fundamental groups classifies continuous maps fixing a point. 
\end{remark}

\subsection{The translation of Theorem \ref{main1} and \ref{main2}}
In this subsection, we translate Theorem \ref{main1} and \ref{main2} into group-theoretic theorems. We also study the section problem for the universal hyperelliptic surface bundle.
\subsubsection{The mapping class groups}
In this subsection, we translate Theorem \ref{main1} and \ref{main2} into group-theoretical theorems.

Let Diff$(S_{g,n})$ denote the orientation-preserving diffeomorphism group of a surface $S_g$ of genus $g>1$ fixing $n$ distinct points $\{x_1,x_2,...,x_n\}\subset S_g$ pointwise. There is a fiber bundle
\begin{equation}
S_g \to \text{UDiff}(S_{g,n})\xrightarrow{u_{g,n}}\text{BDiff}(S_{g,n}),
\label{universal}
\end{equation}
which is universal in the sense that any $S_g$-bundle endowed with $n$ disjoint sections is a pullback of this bundle. Since Diff$(S_{g,n})$ fixes the $n$ points $x_1,x_2,...,x_n$, we associate $n$ points on each fiber, i.e. n disjoint sections of (\ref{universal}) which are denoted by $s_1,s_2,...,s_n$.

Let Diff$(S_{g,\overline{n}})$ denote the orientation-preserving diffeomorphism group of a surface $S_g$ of genus $g>1$ fixing $n$ points $\{x_1,x_2,...,x_n\}\subset S_g$ as a set. There is a fiber bundle
\begin{equation}
S_g \to \text{UDiff}(S_{g,\overline{n}})\xrightarrow{u'_{g,n}} \text{BDiff}(S_{g,\overline{n}}).
\label{universal2}
\end{equation}

Since Diff$_0(S_{g,n})$ and Diff$_0(S_{g,\overline{n}})$ are contractible by Earle and Eells \cite[Theorem 1]{EE}, we have that $\text{BDiff}(S_{g,n})=K(\text{PMod}_{g,n},1)$ and $\text{BDiff}(S_{g,\overline{n}})=K(\text{Mod}_{g,n},1)$.  

Let $\text{PConf}_n(S_g)$ be the space of ordered $n$-tuple of distinct points on $S_g$. There is a natural permutation group $\Sigma_n$-free action on $\text{PConf}_n(S_g)$. Let Conf$_n(S_g):=\text{PConf}_n(S_g)/\Sigma_n$ be the ordered $n$-tuple of distinct points on $S_g$. Let $PB_n(S_g):=\pi_1(\text{PConf}_n(S_g))$ and $B_n(S_g):=\pi_1(\text{Conf}_n(S_g))$ be the $n$-strand ordered and unordered \emph{surface braid groups}. We have the following \emph{Birman exact sequences} describing the monodromy representations of fiber bundle (\ref{universal}) and (\ref{universal2}).
\begin{equation}
1\to PB_n(S_g)\xrightarrow{\text{point pushing}} \text{PMod}_{g,n}\xrightarrow{\pi_{g,n}} \text{Mod}_g\to 1
\label{BES1}
\end{equation}
and
\begin{equation}
1\to B_n(S_g)\xrightarrow{\text{point pushing}}  \text{Mod}_{g,n}\xrightarrow{\pi'_{g,n}} \text{Mod}_g\to 1.
\label{BES2}
\end{equation}

Because of correspondence (\ref{correspondence2}), the classification of continuous sections of fiber bundle (\ref{universal}) and (\ref{universal2}) is the same as the classification of homomorphisms $p$ and $p'$ up to conjugacy that make the following diagrams (\ref{diagram1}) and (\ref{diagram2}) commute. 
  \begin{equation}
\xymatrix{
1 \to PB_n(S_g)\ar[r]\ar@{~)}[d]^R & \text{PMod}_{g,n} \ar[r]^{\pi_{g,n}}\ar@{~)}[d]^{p}& \text{Mod}_g\ar[r]\ar[d]^=&  1 \\
1 \to\pi_1(S_g)\ar[r] & \text{Mod}_{g,1}\ar[r]^{\pi_{g,1}}& \text{Mod}_g \ar[r]& 1}              
\label{diagram1}
\end{equation}
and
  \begin{equation}
\xymatrix{
1 \to B_n(S_g)\ar[r]\ar@{~)}[d]^{R'} & \text{Mod}_{g,n} \ar[r]^{\pi_{g,n}'}\ar@{~)}[d]^{p'}& \text{Mod}_g\ar[r]\ar[d]^=&  1 \\
1 \to\pi_1(S_g)\ar[r] & \text{Mod}_{g,1}\ar[r]^{\pi_{g,1}}& \text{Mod}_g \ar[r]& 1.}          
\label{diagram2}
\end{equation}
  
For $p$ and $p'$ satisfying diagrams (\ref{diagram1}) and (\ref{diagram2}), denote by $R$ and $R'$ the restrictions of $p$ and $p'$ on the subgroups $PB_n(S_g)$ and $B_n(S_g)$. Let $\text{PMod}_{g,n}\xrightarrow{p_{g,n,i}} \text{Mod}_{g,1}$ be the forgetful homomorphism that forgets the fixed points $\{x_1,...,\hat{x_i},...,x_n\}$. Theorem \ref{main1} is thus equivalent to the following Theorem.
\begin{thm}
For $g>2$ and $n\ge 0$, every homomorphism $p$ satisfying diagram (\ref{diagram1}) is conjugate to $p_{g,n,i}$ for some $i$ by an element $A$ in $\pi_1(S_g)$.
\label{main3}
\end{thm}

Theorem \ref{main2} is thus equivalent to the following.
\begin{thm}
For $g>1$ and $n>1$, there is no homomorphism $p'$ satisfying diagram (\ref{diagram2}).
\label{main4}
\end{thm}

\subsubsection{The hyperelliptic mapping class groups}
In this subsection, we translate the section problem of the hyperelliptic surface bundle into a group-theoretical statement.

Let $\tau$ be the involution as in the following figure.
\begin{figure}[H]
\centering
 \includegraphics[scale=0.3]{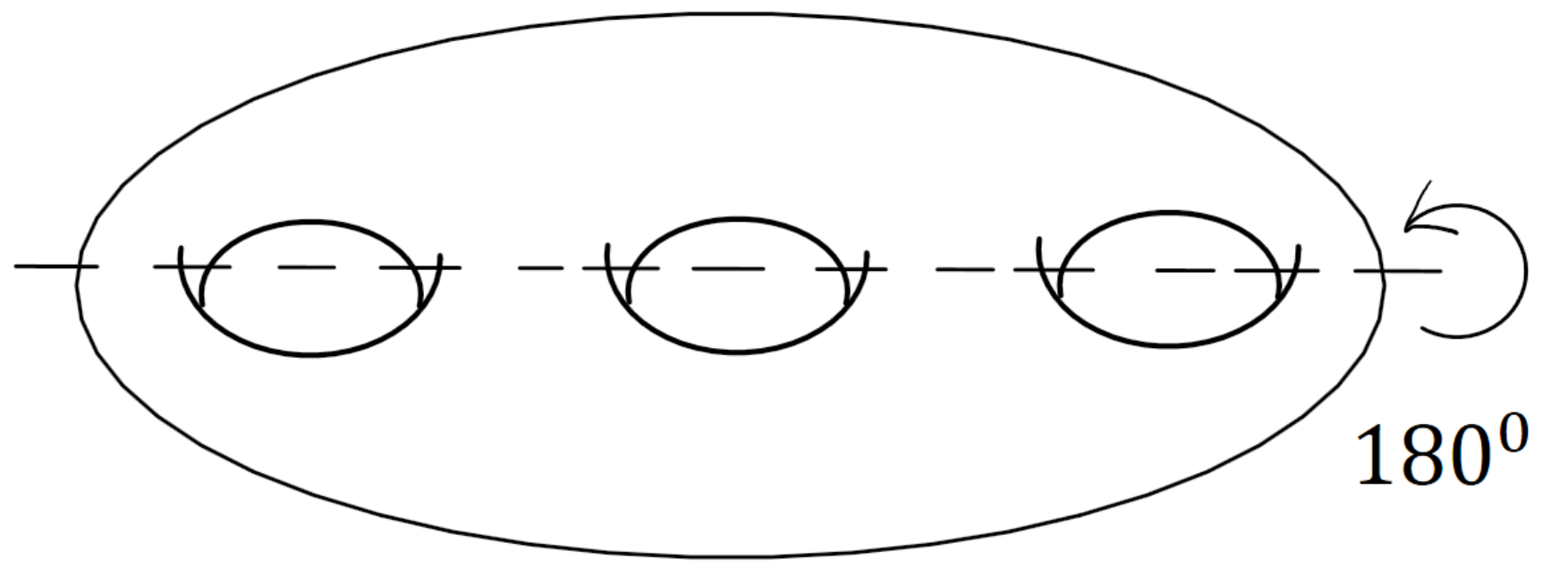}
 \caption{Hyperelliptic involution $\tau$ for $g=3$ case}
  \label{figure1}
\end{figure}

Let ${\cal H}_g$ be the \emph{hyperelliptic mapping class group}, i.e. the subgroup of Mod$_g$ consisting of all the mapping classes that are commutative with $\tau$. Denote by ${\cal H}_{g,n}$ (resp. ${\cal PH}_{g,n}$) the hyperelliptic mapping class group fixing $n$ points as a set (resp. pointwise), i.e. they satisfy the following pullback diagrams.
\[
\xymatrix{
{\cal H}_{g,n} \ar[r]\ar[d]   \pb &{\cal H}_g\ar[d]\\
\text{Mod}_{g,n} \ar[r] & \text{Mod}_g}
\qquad\text{and}\qquad
 \xymatrix{{\cal PH}_{g,n}\ar[r]\ar[d]   \pb &{\cal H}_g\ar[d]\\
\text{PMod}_{g,n}\ar[r] & \text{Mod}_g.}
\]

Let ${\cal BPH}_{g,n}=K({\cal PH}_{g,n},1)$ be the \emph{pure universal hyperelliptic space} fixing $n$ punctures pointwise and let 
\begin{equation}
S_g\to {\cal UPH}_{g,n}\xrightarrow{Hu_{g,n}}{\cal BPH}_{g,n}
\label{THB2}
\end{equation}
 be the \emph{pure universal hyperelliptic bundle}, i.e. the bundle that corresponds to the monodromy $\rho_{H,g,n}:{\cal PH}_{g,n}\to \text{PMod}_{g,n}$. Surface bundle (\ref{THB2}) classifies smooth $S_g$-bundle equipped with a $\tau$-action and $n$ unordered points on each fiber. For any section $s$, we could generate another section $t=\tau(s)$. Denote by $Hs_i$ the pullback of $s_i$ as a section of bundle (\ref{THB2}) and denote by $Ht_i$ their hyperelliptic conjugates.

Let ${\cal BH}_{g,n}=K({\cal H}_{g,n},1)$ be the \emph{universal hyperelliptic space} fixing $n$ punctures as a set and let 
\begin{equation}
S_g\to {\cal UH}_{g,n}\xrightarrow{Hu'_{g,n}} {\cal BH}_{g,n}
\label{THB1}
\end{equation}
 be the \emph{universal hyperelliptic bundle}, i.e. the bundle that corresponds to the monodromy $\rho_{H,g,n}': {\cal H}_{g,n}\to \text{PMod}_{g,n}$. Surface bundle (\ref{THB1}) classifies smooth $S_g$-bundle equipped with a $\tau$-action and $n$ unordered points on each fiber. We have the following classification of sections for bundles (\ref{THB2}) and (\ref{THB1}).

\begin{thm}[{\bf Hyperelliptic case}]

1) For $n\ge 0$ and $g>1$, every section of the universal hyperelliptic undle (\ref{THB2}) is homotopic to $Hs_i$ or $Ht_i$ for some $i\in \{1,2,...,n\}$.  \\
2) For $n>1$ and $g>1$, the universal hyperelliptic bundle (\ref{THB1}) has no sections.
\label{Hyper}
\end{thm} 
By correspondence (\ref{correspondence2}), we can translate Theorem \ref{Hyper} into the following group-theoretical statement. Let ${\cal PH}_{g,n} \xrightarrow{{H\pi_{g,n}}} {\cal H}_g$ and ${\cal H}_{g,n} \xrightarrow{{H\pi'_{g,n}}} {\cal H}_g$ be the forgetful maps forgetting the punctures. Let ${\cal H}_{g,n}\xrightarrow{Hp_{g,n,i}} \text{Mod}_{g,1}$ be the forgetful homomorphism forgetting the fixed points $\{x_1,...,\hat{x_i},...,x_n\}$. 

\begin{prop}
1) Every homomorphism $p$ satisfying the following diagram is either conjugate to the forgetful homomorphism $Hp_{g,n,i}$ by an element in ${\cal PH}_{g,n}$ or factors through $H\pi_{g,n}$, i.e. there exists $f$ such that $p=f\circ H\pi_{g,n}$.
\begin{equation}
\xymatrix{
1 \to PB_n(S_g)\ar[r]\ar@{~)}[d]^R & {\cal PH}_{g,n} \ar[r]^{H\pi_{g,n}}\ar@{~)}[d]^p& {\cal H}_g\ar[r]\ar[d]^{\rho_{H,g}}&  1 \\
1 \to\pi_1(S_g)\ar[r] &\text{\normalfont Mod}_{g,1}\ar[r]^{H\pi_{g,1}}& \text{\normalfont Mod}_g \ar[r]& 1.}         
\label{diagram11}
\end{equation}

2) For $n>1$, every homomorphism $p'$ satisfying the following diagram factors through $H\pi_{g,n}'$, i.e. there exists $f'$ such that $p'=f'\circ H\pi_{g,n}'$
  \begin{equation}
\xymatrix{
1 \to B_n(S_g)\ar[r]\ar@{~)}[d]^{R'} & {\cal H}_{g,n} \ar[r]^{H\pi_{g,n}'}\ar@{~)}[d]^{p'}& {\cal H}_g\ar[r]\ar[d]^{\rho_{H,g}}&  1 \\
1 \to\pi_1(S_g)\ar[r] &\text{\normalfont Mod}_{g,1}\ar[r]^{H\pi_{g,1}}&\text{\normalfont Mod}_g \ar[r]& 1.}    
\label{diagram21}
\end{equation}
\label{HP}
\end{prop}
\begin{proof}[\bf Proof of Theorem \ref{Hyper} assuming Proposition \ref{HP}]
By Proposition \ref{HP}, $p$ has the following two cases.\\
\\
{\bf Case 1: $p$ is conjugate to the forgetful homomorphism $Hp_{g,n,i}$ by an element $A\in {\cal PH}_{g,1}$.}

 By the commutativity of diagram (\ref{diagram11}), the mapping class $H\pi_{g,n}(A)$ is in the center of ${\cal H}_g$. Since Center$({\cal H}_g)=\big\langle \tau \big\rangle$, e.g. see \cite[Section 3.4 and Section 9.4]{BensonMargalit}, we have that ${\cal H}\pi_{g,n}(A)=1 \text{ or } \tau$, which represent section $Hs_i$ and $Ht_i$. \\
\\
 {\bf Case 2: $p$ factors throught $H\pi_{g,n}$.}

To prove the result, we only need to show that the exact sequence
\begin{equation}
1\to \pi_1(S_g)\to {\cal H}_{g,1}\to {\cal H}_g\to 1
\label{HU}
\end{equation}
does not split. The following finite order mapping class $\sigma$ is commutative with $\tau$.
\begin{figure}[H]
\centering
 \includegraphics[scale=0.3]{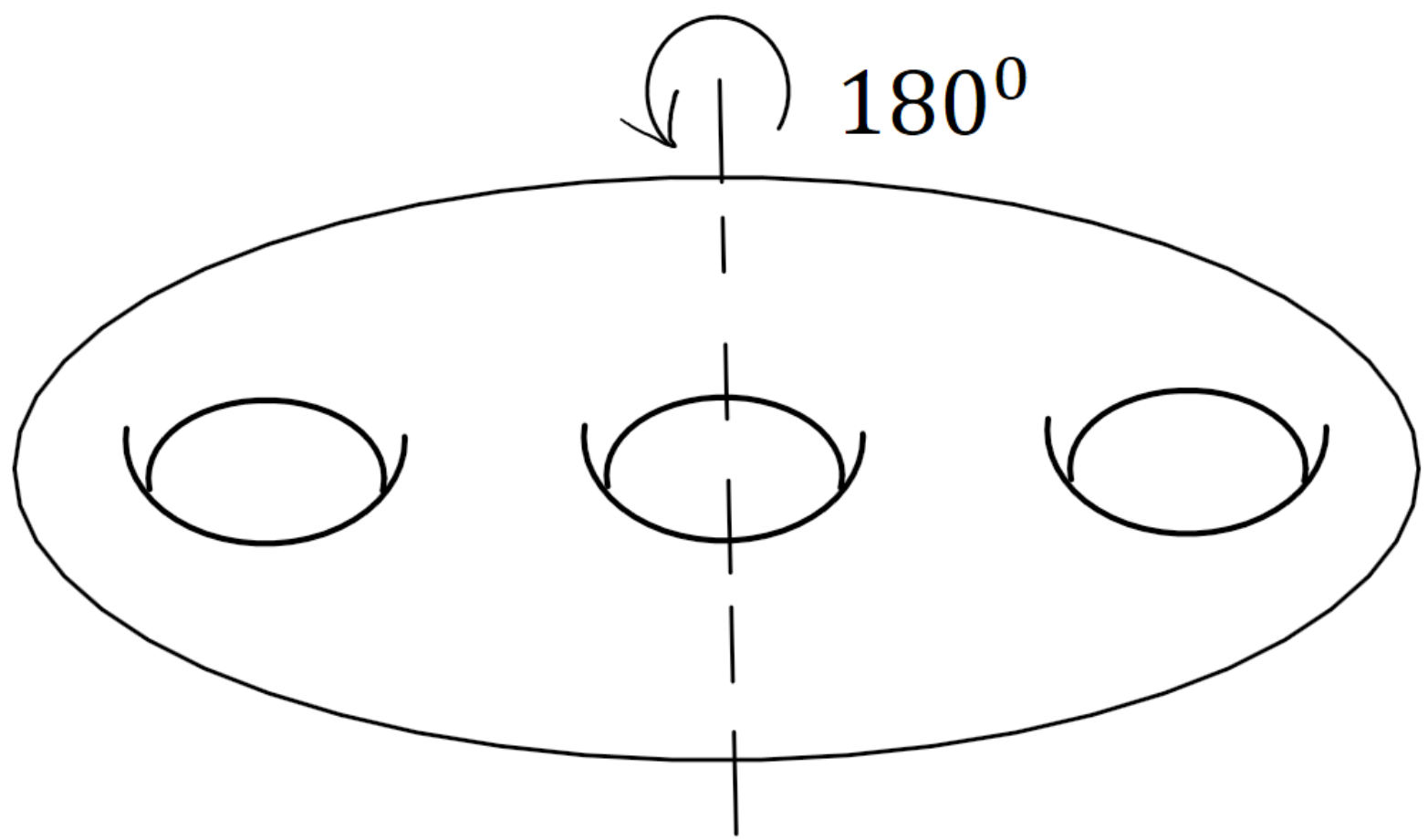}
 \caption{Torsion mapping class $\sigma$ for $g=3$ case}
  \label{figure1}
\end{figure}

In ${\cal H}_g$, mapping classes $\tau$ and $\sigma$ generate a $\mathbb{Z}/2\times \mathbb{Z}/2$ subgroup; this contradicts the fact that every finite subgroup of Mod$_{g,1}$ is cyclic. Therefore exact sequence (\ref{HU}) does not split.

\end{proof}

\section{The classification of homomorphisms $PB_n(S_g)\xrightarrow{R} \pi_1(S_g)$}
This section is divided into three parts. We first compute $H^*(\text{PConf}_n(S_g));\mathbb{Q})$, then study an algebraic property of $H^*(\text{PConf}_n(S_g));\mathbb{Q})$, Finally we use the computation and the property to prove Theorem \ref{sur}. The key idea is an argument of \cite{FEAJohnson} that there is some cohomological constraint on the existence of homomorphisms $PB_n(S_g)\xrightarrow{R} \pi_1(S_g)$. We assume throughout that $g>1$ and $n>0$.
\subsection{The computation of $H^*(\text{PConf}_n(S_g);\mathbb{Q})$}
In this subsection, we compute $H^*(\text{PConf}_n(S_{g,p});\mathbb{Q})$. Let $S_g^n$ be the product of $n$ copies of $S_g$. There is a natural embeddings $\text{PConf}_n(S_g)\subset S_g^n$. Let $p_i:\text{PConf}_n(S_g)\to S_g$ be the projection onto the $i$th component. Denote by $\triangle_{ij}\approx S_g^{n-1}\subset S_g^n$ be the $ij$th diagonal subspace of $S_g^n$, i.e. $\triangle_{ij}$ consists of points in $S_g^n$ such that the $i$th and $j$th coordinates are equal.. Let $H_i:=p_i^*H^1(S_g;\mathbb{Q})$ and let $[S_g]$ be the fundamental class in $H^2(S_g;\mathbb{Q})$. 
\begin{lem}
1) For $g>1$ and $n>0$,
\begin{equation}
H^1(\text{\normalfont PConf}_n(S_{g});\mathbb{Q})\cong H^1(S_{g}^n;\mathbb{Q})\cong \bigoplus_{i=1}^{n}H_i.
\label{EX1}
\end{equation}
2)We have an exact sequence

\begin{equation}
1\to \oplus_{1\le i<j\le n}\mathbb{Q}[G_{ij}] \xrightarrow{\phi} H^2(S_{g}^n;\mathbb{Q})\cong \bigoplus_{i=1}^{n}\mathbb{Q}p_i^*[S_g]\oplus \bigoplus_{i\neq j} H_i\otimes H_j\xrightarrow{Pr} H^2(\text{PConf}_n(S_{g});\mathbb{Q}),
\label{EX2}
\end{equation}
where $\phi(G_{ij})=[\triangle_{ij}]\in H^2(S_{g}^n;\mathbb{Q})$ is the Poincar\'e dual of the diagonal $\triangle_{ij}\subset S_g^n$.
\end{lem}
\begin{proof}
There is a graded-commutative $\mathbb{Q}$-algebra $[G_{ij}]$ defined in \cite[Theorem 1]{totaro}, where the degree of the generators $G_{ij}$ is 1. By Totaro \cite[Theorem 1]{totaro}, there is a spectral sequence $E_2^{p,q}=H^p(S_g^n;\mathbb{Q})[G_{ij}]^q$ converging to $H^*(\text{PConf}_n(S_g);\mathbb{Q})$. Since we only compute $H^1$ and $H^2$, the differential involved is $d_2:E_2^{0,1}=H^0(S_g^n;\mathbb{Q})[G_{ij}]\to E_2^{2,0}=H^2(S_g^n;\mathbb{Q})$. Let $[\triangle_{ij}]\in H^2(S_{g}^n;\mathbb{Q})$ be the Poincar\'e dual of $\triangle_{ij}\subset S_g^n$. By \cite[Theorem 2]{totaro}, the differential $d_2(G_{ij})=[\triangle_{ij}]$. All the isomorphisms in the lemma are coming from the K\"unneth formula.
\end{proof}

Let $\{a_k,b_k\}_{k=1}^g$ be a symplectic basis for $H^1(S_g;\mathbb{Q})$. For $1\le i,j \le m$, we denote 
\[M_{i,j}=\sum_{k=1}^{n} p_i^*a_k \otimes p_j^*b_k-p_i^*b_k\otimes p_j^*a_k.\]

The following lemma describes $[\triangle_{ij}]\in H^2(S_{g}^n;\mathbb{Q})\cong \bigoplus_{i=1}^{n}\mathbb{Q}p_i^*[S_g]\oplus \bigoplus_{i\neq j} H_i\otimes H_j$.
\begin{lem}
The diagonal element $[\triangle_{ij}]=p_i^*[S_{g}]+p_j^*[S_{g}]+M_{ij}\in \bigoplus_{i=1}^{n}\mathbb{Q}p_i^*[S_g]\oplus \bigoplus_{i\neq j} H_i\otimes H_j\cong H^2(S_{g}^n;\mathbb{Q})$. 
\end{lem}
\begin{proof}
This is classical. See \cite[Section 11]{CC}.
\end{proof}

\subsection{A property of the cup product structure of $H^*(\text{PConf}_n(S_g);\mathbb{Q})$}
In this subsection, we talk about a property of the cup product $H^1\otimes H^1\to H^2$ for PConf$_n(S_g)$. 
\begin{defn}
We call an element $x=(x^1,...,x^n) \in \bigoplus_{i=1}^{n}H_i=H^1(\text{PConf}_n(S_g);\mathbb{Q})$ a {\em crossing element} if $\#\{i:x^i\neq0\}>1$, i.e. $x\notin H_i$ for any $i$.
\end{defn}

\begin{lem}
Let $x=(x^1,...,x^n)$ and $y=(y^1,...,y^n)$ be two elements in $H^1(\text{PConf}_n(S_g);\mathbb{Q})$. Suppose that $x$ or $y$ is a crossing element. If $x\smile y=0\in H^2(\text{PConf}_n(S_g);\mathbb{Q})$, then $x$ and $y$ are proportional, i.e. $\lambda x=\mu y$ for some constants $\lambda\in \mathbb{Q}$ and $\mu\in \mathbb{Q}$.
\label{crossing}
\end{lem}
\begin{proof}The multiplication of $x$ and $y$ is the following:
\[
x\smile y=x^1\smile y^1+...+x^n\smile y^n+\sum_{i\neq j}(x^i\otimes y^j-y^i\otimes x^j)\in \bigoplus_{i=1}^{n}\mathbb{Q}p_i^*[S_g]\oplus \bigoplus_{i\neq j} H_i\otimes H_j\xrightarrow{Pr} H^2(\text{PConf}_n(S_{g});\mathbb{Q}).
\]
By $x\smile y=0\in H^2(\text{PConf}_n(S_g);\mathbb{Q})$ and exact sequence (\ref{EX2}), we have the following equality in $\bigoplus_{i=1}^{n}\mathbb{Q}p_i^*[S_g]\oplus  \bigoplus_{i\neq  j}H_i\otimes H_j$:
\[x^1\smile y^1+...+x^n\smile y^n+\sum_{i\neq j}(x^i\otimes y^j-y^i\otimes x^j)=\sum k_{i,j}[\triangle_{i,j}]=\sum k_{i,j}(p_i[S_g]+p_j[S_g]+M_{i,j}).\]

By the independence of all the terms in $\mathbb{Q}^n\oplus  \bigoplus_{i\neq  j}H_i\otimes H_j$, we have 
\[x^i\otimes y^j-y^i\otimes x^j=k_{i,j}M_{i,j}\text{     for all $i,j$}.\]

If $x^i$ and $y^i$ are proportional in $H_i$, since $g>1$, we have at least 4 terms in $M_{i,j}$, we don't have enough basis to span our $M_{i,j}$. If $x^i$ and $y^i$ are independent in $H_i$, since $g>1$, we have $x^i\otimes y^j-y^i\otimes x^j\neq M_{i,j}$. Therefore $k_{i,j}=0$
and $x^i\otimes y^j-y^i\otimes x^j=0\in H_i\otimes H_j$. Assume without loss of generality that $x$ a crossing element and $x_1\neq 0$ and $x_2\neq 0$. We break the proof into the following cases.

{\bf\boldmath Case 1) $y^1\neq 0$ and $y^1$ is not proportional to $x^1$}

$x^1\otimes y^j=y^1\otimes x^j\in H_1\otimes H_j$ implies that $y^j=0$ and $x^j=0$ for $j$. However $x_2\neq 0$. Therefore this case is invalid.

{\bf \boldmath Case 2) $y^1\neq 0$ and $\lambda x^1=\mu y^1$}

$x^1\otimes y^j=y^1\otimes x^j\in H_1\otimes H_j$ implies that $\lambda x^j=\mu y^j$ for all $j$, which verifies oi=1ur lemma that $x$ and $y$ are proportional.

{\bf \boldmath Case 3) $y^1=0$}

$x^1\otimes y^j=y^1\otimes x^j\in H_1\otimes H_j$ implies that $y^j=0$ for all $j$. This means $y=0$ therefore $x$ and $y$ are also proportional.
\end{proof}

\subsection{The proof of Theorem \ref{sur}}
In this subsection, we use the computation of $H^*(\text{PConf}_n(S_g);\mathbb{Q})$ and Lemma \ref{crossing} to prove Theorem \ref{sur}. Let $p_{i*}:PB_n(S_g)\to \pi_1(S_g)$ be the induced map on the fundamental groups of $p_i:\text{PConf}_n(S_g)\to S_g$.
\begin{lem}
Let $F_h$ be a free group of $h$ generators and let $S_r$ be a surface of genus $r$.
If we have a surjective homomorphism $PB_n(S_g)\xrightarrow{S} \Gamma$ when $\Gamma=F_h$ with $h>1$ or $\Gamma=\pi_1(S_r)$ with $r>1$, and we also have $p_i^*(H^1(S_g;\mathbb{Q}))\cap S^*(H^1(\Gamma;\mathbb{Q}))\neq \{0\}$, then $S$ factors through $p_{i*}$ for some $i$.
\label{noint}
\end{lem}
\begin{proof}
The proof of this lemma uses the same idea as \cite{FEAJohnson}. The method can also be found in \cite[Lemma 3.3 and 3.4]{NickFibering}. If there is a common nonzero cohomology element $S^*(x)=p_{i*}^*(y)$ for $x\in H^1(F_h;\mathbb{Q})$ and $y\in H^1(\pi_1(S_g);\mathbb{Q})$, we have the following commutative diagram by the identification $H^1(\_\_;\mathbb{Q})\cong \text{Hom}(\_\_,\mathbb{Q})$.
\[
\xymatrix{
PB_n(S_g)\ar[r]^-S\ar[d]^{p_{i*}} & F_h\ar[d]^x\\
\pi_1(S_g)\ar[r]^y & \mathbb{Q}}
\]
Let $K$ be the kernel of $p_{i*}$, which is a finitely generated normal subgroup of $PB_n(S_g)$. The image of $S(K)$ is also a finite generated normal subgroup of $\pi_1(F_h)$. However every finitely generated normal subgroup of $F_h$ either is finite index or is trivial. For a surface group of genus $r$ case, any nontrivial finitely-generated normal subgroup of $\pi_1(S_r)$ has finite index; see Property (D6) in \cite{FEAJohnson}. If $S(K)\subset F_h$ has finite index, then after composing with $x$, the image $x\circ S(K)$ won't be trivial in $\mathbb{Q}$; however $K$ is the kernel of $p_{i*}$ so $x\circ S(K)=y\circ p_{i*}(K)=1$.

If the $S(K)=1$, then $S$ factors through $p_{i*}$.\end{proof}

To prove Theorem \ref{sur}, we have to include a lemma talking about the possible image of the homomorphism.
\begin{lem}
Every finitely generated subgroup of $\pi_1(S_g)$ is either finitely generated free group $F_h$ or surface group $\pi_1(S_r)$ with $r\ge g$. When $r=g$, the subgroup is the whole group $\pi_1(S_g)$.
\label{possible}
\end{lem}
\begin{proof}
A subgroup $G$ of $\pi_1(S_g)$ corresponds to a cover $S$ of $S_g$ such that $G=\pi_1(S)$. If $S$ is noncompact, then $\pi_1(S)$ is free group. If $S$ is compact, it is a finite cover. Therefore $\pi_1(S)=S_r$ for some $r$. If $S$ is a $k$-cover. The Euler characteristic is multiplicative under cover, thus $\chi(S_r)=k\chi(S_g)$. If $g>1$ and $k>1$, we have $r>g$. If $n=1$, this is trivial cover.
\end{proof}

\begin{proof}[Proof of Theorem \ref{sur}] 
Let  $R:PB_n(S_g)\xrightarrow{} \pi_1(S_g)$ be a homomorphism. By Lemma \ref{possible}, if Im$(R)\cong \mathbb{Z}$, the image has to be $F_h$ with $h>1$ or $\pi_1(S_r)$ with $r\ge g$. Furthermore, if $S$ does not factor through $p_{i*}$ for some $i$, then by Lemma \ref{noint}, $S^*(H^1(\text{Im}(R);\mathbb{Q}))$ does not intersect nontrivially with any $H_i$. This means that all nonzero elements of $S^*(H^1(\text{Im}(R);\mathbb{Q}))$ are crossing elements. However $r\ge g>1$ and $h>1$ mean that there are two crossing elements $x$ and $y$ in $S^*(H^1(\text{Im}(R);\mathbb{Q}))$ that are independent and their cup product is zero. Lemma \ref{crossing} tells us that this is impossible, which successfully proves 2) of Theorem \ref{sur}.

Now to prove 1), we have a surjection homomorphism $PB_n(S_g)\xrightarrow{p_{i*}} \pi_1(S_g)\xrightarrow{A} \pi_1(S_g)$. However surface groups are Hopfian which means that a surjective self homomorphism between the surface group $\pi_1(S_g)$ must be an automorphism. Therefore $A$ is an automorphism, which concludes the proof of 1) in Theorem \ref{sur}. 

\end{proof}

\section{Applications of Theorem \ref{sur}}
In this section, we apply Theorem \ref{sur} to the study of section problems of universal surface bundles. 
\subsection{The proof of Theorem \ref{main3}}
Since we already established all possible homomorphisms $R$ in Theorem \ref{sur}, the key idea of extending the homomorphism to $\text{PMod}_{g,n}$ is that it has to be equivariant with the action of $\text{Mod}_g$. We then use homology to rule out other possibilities.

\begin{defn}
Let a subspace $H\subset H^1(\text{PConf}_n(S_g);\mathbb{Q})\cong \bigoplus_{i=1}^{n} H_i$ be an \emph{isotropic subspace} if any $a,b\in H$, we have $a\smile b=0\in H^2(\text{PConf}_n(S_g);\mathbb{Q})$.
\end{defn}

The following lemma is needed in the proof.
\begin{lem}
Mod$_g$ does not fix any isotropic subspace of $H^1(\text{PConf}_n(S_g);\mathbb{Q})$.
\end{lem}
\begin{proof}
If there exists a crossing element $x\in H$, because of Lemma \ref{crossing}, we know that $x\smile y=0$ if and only if $y$ is proportional to $x$. Therefore if $H$ is isotropic, $H=\mathbb{Q}x\subset H^1(\text{PConf}_n(S_g);\mathbb{Q})$.

If dim$(H)>1$, then $H$ does not contain crossing elements by Lemma \ref{crossing}. In this case, if there exist $x,y \in H$ and $i\neq j\in \{1,2,...,n\}$ such that $x\neq 0\in H_i$ and $y\neq 0\in H_j$, we would have $x+y$ a crossing element. Therefore there exists $i$ such that $H\subset H_i$.

Mod$_g$ acts on $H^1(\text{PConf}_n(S_g);\mathbb{Q})\cong \bigoplus_{i=1}^{n} H_i$ by acting on each component. We know that the action of Mod$_g$ on $H^1(S_g;\mathbb{Q})$ does not fix any isotropic subspace, therefore if $H\subset H_i$, Mod$_g$ does not fix $H$. If dim$(H)=1$, Mod$_g$ also does not fix it.
\end{proof}

Now we finish the proof of Theorem \ref{main3}.
\begin{proof}[\bf Proof of Theorem \ref{main3}]
If we can extend $R$, then for $e\in \text{PMod}_{g,n}$ and $f\in PB_n(S_g)$ we have $R(efe^{-1})=p(e)R(f)p(e)^{-1}$. The action of $\text{PMod}_{g,n}$ and $\text{Mod}_{g,n}$ on $PB_n(S_g)$ and $B_n(S_g)$ are given by conjugation in the exact sequence (\ref{BES1}) and (\ref{BES2}), respectively. Therefore we have a commutative diagram:
\[
\xymatrix{  PB_n(S_g)\ar[r]^{e}\ar[d]^R & PB_n(S_g)\ar[d]^R \\
\pi_1(S_g)\ar[r]^{{p(e)}} & \pi_1(S_g).}
\]

Since both $e$ and $p(e)$ are isomorphisms of groups, we have that $\text{Im}(R)=\text{Im}(R\circ e)$. This gives us the following diagram:
\begin{equation}
\xymatrix{  PB_n(S_g)\ar[r]^{e}\ar[d]^R & PB_n(S_g)\ar[d]^R \\
\text{Im}(R)\ar[r]^{{p(e)}} & \text{Im}(R).}
\label{Restricting}
\end{equation}

Because of Lemma \ref{possible}, we know that we have 4 possibilities for $\text{Im}(R)$: $F_h$ for $h=0$, $h>0$ and $\pi_1(S_r)$ for $r=g$ or $r>g$.
Now, we go over all possibilities.\\
\\
\noindent
{\large \bf Case 1) $\text{Im}(R)=1$}

In this case, we have a homomorphism $\text{PMod}_{g,n}/PB_n(S_g)=\text{Mod}_g\to \text{Mod}_{g,1}$. However, the $n=0$ case has already been proved, for example in \cite[Corollary 5.11]{BensonMargalit}.\\
\\
\noindent
{\large \bf Case 2) $\text{Im}(R)=F_h$, while $h>0$}

We have the following diagram.
\begin{equation}
\xymatrix{  PB_n(S_g)\ar[r]^e\ar[d]^R & PB_n(S_g)\ar[d]^R \\
F_h\ar[r]^{p(e)} & F_h}
\label{com}
\end{equation}
For every $e\in \text{PMod}_{g,n}$, diagram (\ref{com}) means that $R^*(H^1(F_h;\mathbb{Q}))\subset H^1(S_g;\mathbb{Q})$ has to be fixed under the action of $\text{Mod}_g$. This is impossible because $R^*(H^1(F_h;\mathbb{Q}))\subset H^1(\text{PConf}_n(S_g);\mathbb{Q})$ is an isotropic subspace of $H^1(\text{PCon}f_n(S_g);\mathbb{Q})$, but $\text{Mod}_g$ does not fix any isotropic subspace.\\
\\
\noindent
{\large \bf Case 3) $\text{Im}(R)=\pi_1(S_g)$}

If $R$ is one of the forgetful homomorphism $p_i$, then for $e\in \text{PMod}_{g,n}$ and $f\in PB_n(S_g)$,
\[
p_i(efe^{-1})=p(e)p_i(f)p(e)^{-1}.
\]
We get that $p_{i}(e)p_i(f)p_i(e)^{-1}=p(e)p_i(f)p(e)^{-1}$. Therefore $p(e)^{-1}p_i(e)$ commutes with $p_i(f)$ for any $f\in PB_n(S_g)$. The image of $p_i$ on $PB_n(S_g)$ is the whole group $\pi_1(S_g)$. Therefore, $p(e)^{-1}p_i(e)\in \text{Mod}_{g,1}$ commutes with the subgroup $\pi_1(S_g)$. However, the centralizer of $\pi_1(S_g)<Mod_{g,1}$ is $1$, so we get that $p(e)^{-1}p_i(e)=1\in \text{Mod}_{g,1}$. This tells us that $p=p_i$. 

If R is one of the forgetful homomorphism $p_i$ post-composing with an automorphism $A$, with a similar argument as above, we get that $p(e)= Ap_i(e)A^{-1}$. Considering that the images of $Ap_i(e)A^{-1}$ and $p_i(e)$ have to be equal in Mod$_g$ for any $e$, we have $Ap_i(e)=p_i(e)A$ for any $e\in \text{Mod}_g$. Therefore, we have $A\in \text{Center} (\text{Mod}_g)$. For $g>2$, $\text{Center}(\text{Mod}_g)=1$, therefore we have $A\in \pi_1(S_g)$. For $g=2$, we could have $A=\tau$.\\
\\
\noindent
{\large \bf Case 4) $\text{Im}(R)=\pi_1(S_r)$ while $r>g$}

Because of Lemma \ref{sur}, $R$ factors through $p_i$. However there is no surjective homomorphism from $\pi_1(S_g)\to \pi_1(S_r)$ since $\text{Rank}(H^1(S_r;\mathbb{Q}))>\text{Rank}(H^1(S_g;\mathbb{Q}))$.

\end{proof}

\subsection{The proof of Theorem \ref{main4}}
In this section, we begin with the proof of Theorem \ref{main4}

\begin{lem}
\[H^1(\text{Conf}_n(S_g);\mathbb{Q})\cong H^1(S_g;\mathbb{Q})\]
and the image of $H^1(\text{Conf}_n(S_g);\mathbb{Q})\to H^1(\text{PConf}_n(S_g);\mathbb{Q})$ is equal to the image of the diagonal map, i.e. $x\to (x,x,...,x)\in H^1(\text{PConf}_n(S_g);\mathbb{Q})\cong \bigoplus_{i=1}^n H_i$.
\end{lem}
\begin{proof}
Since $\text{Conf}_n(S_g)=\text{PConf}_n(S_g)/\Sigma_n$, we can use the transfer map to get 
\[H^1(\text{Conf}_n(S_g);\mathbb{Q})=H^1(\text{PCon}f_n(S_g);\mathbb{Q})^{\Sigma_n}.\]
 It is not hard to see that $H^1(\text{PConf}_n(S_g);\mathbb{Q})^{\Sigma_n}$ is the diagonal subspace.
\end{proof}

\begin{proof}[\bf Proof of Theorem \ref{main4}]
If we have a homomorphism $p'$ in the diagram in Theorem \ref{main4}, after composing an injection $\text{PMod}_{g,n}\xrightarrow{i} \text{Mod}_{g,n}$, we get a homomorphism $p$ as in Theorem \ref{main3}. Let $C_A$ denote the conjugate by $A\in\pi_1(S_g)$. We have already proved Theorem \ref{main3} that $p'\circ i=C_A\circ p_i \in \pi_1(S_g)$ for some $i$ and $A\in \pi_1(S_g)$. Restricting to the kernel of $\pi_{g,n}$ and $\pi_{g,n}'$, the following diagram holds.
\[
\xymatrix{ 
PB_n(S_g)\ar[r]^{i}\ar[d]^{p_i}& B_n(S_g)\ar[d]^{R'} \\
\pi_1(S_g)\ar[r]^{C_A} & \pi_1(S_g)}
\]

The image of $H^1(S_g;\mathbb{Q})\xrightarrow{(C_A\circ p_i)^*} H^1(\text{PConf}_n(S_g);\mathbb{Q})$ is $H_i$; however the image of $H^1(\text{Conf}_n(S_g);\mathbb{Q})\to H^1(\text{PConf}_n(S_g);\mathbb{Q})$ as described in the previous lemma is the diagonal. Thus this is a contradiction.
\end{proof}

\subsection{The hyperelliptic case}
In this subsection, we prove Proposition \ref{HP}. The proof follows the same argument as the proof of Theorem \ref{main3}. The following lemma is a key ingredient in the proof.
\begin{lem}
The action of ${\cal H}_g$ on $H^1(S_g;\mathbb{Q})$ does not preserve any isotropic subspace.
\label{fact}
\end{lem}
\begin{proof}
Let Sp$_{2g}(\mathbb{Z})[m]$ be the kernel of the map Sp$_{2g}(\mathbb{Z})\to Sp_{2g}(\mathbb{Z}/m)$. By \cite[Theorem 3.3]{brendle2015level}, the image of the monodromy representation $\rho_{s}:{\cal H}_g\to \text{Sp}_{2g}(\mathbb{Z})$ contains Sp$_{2g}(\mathbb{Z})[2]$. Since Sp$_{2g}(\mathbb{Z})[2]$ is a finite index subgroup of Sp$_{2g}(\mathbb{Z})$, we only need to show that Sp$_{2g}(\mathbb{Z})[2]$ does not preserve any isotropic subspace. We prove a stronger result that the stabilizer of any isotropic subspace of $H^1(S_g;\mathbb{Q})$ has infinite index in Sp$_{2g}(\mathbb{Z})$.

For an isotropic subspace $H\subset H^1(S_g;\mathbb{Q})$, let $\text{Stab}_H(\text{Sp}_{2g}(\mathbb{Z}))$ be the stabilizer of $H$ in $\text{Sp}_{2g}(\mathbb{Z})$ and let $\text{Orb}_H(\text{Sp}_{2g}(\mathbb{Z}))$ be the orbit of $H$ under the action of $\text{Sp}_{2g}(\mathbb{Z})$. We have the following equation:
\[
[\text{Sp}_{2g}(\mathbb{Z}):\text{Stab}_H(\text{Sp}_{2g}(\mathbb{Z}))]\cong \text{Orb}_H(\text{Sp}_{2g}(\mathbb{Z})).
\]
Since the order of the $\text{Orb}_H(\text{Sp}_{2g}(\mathbb{Z}))$ is infinite, we have that $\text{Stab}_H(\text{Sp}_{2g}(\mathbb{Z}))$ has infinite index in $\text{Sp}_{2g}(\mathbb{Z})$. This concludes the proof since the $\rho_s({\cal H}_g)\subset \text{Sp}_{2g}(\mathbb{Z})$ has finite index. 

\end{proof}
The proof is identical to the proof of Theorem \ref{main3} with the help of Lemma \ref{fact}. 
\begin{proof}[\bf The proof of Proposition \ref{HP}]
The proof of Case 1), 3), and  4) are the same. Case 2) needs the fact that ${\cal PH}\pi_{g,n}$ does not preserve any isotropic subspace in $H^1(S_g;\mathbb{Q})$ which can be deduced by Lemma \ref{fact} that  ${\cal H}\pi_g$ does not preserve any isotropic subspace in $H^1(S_g;\mathbb{Q})$. 
\end{proof}

\bibliography{citing}{}
\vspace{5mm}
\hfill \break
Dept. of Mathematics, University of Chicago

E-mail: chenlei@math.uchicago.edu

\end{document}